\documentclass{amsart}
\usepackage{amssymb, amscd, ascmac}
\usepackage{pb-diagram}

\newtheorem{theorem}{Theorem}[section]
\newtheorem{proposition}{Proposition}[section]
\newtheorem{lemma}{Lemma}[section]

\theoremstyle{definition}

\makeatletter
\let\c@proposition=\c@theorem
\let\c@lemma=\c@theorem
\let\c@corollary=\c@theorem
\let\c@conjecture=\c@theorem
\let\c@definition=\c@theorem
\let\c@remark=\c@theorem
\let\c@assumption=\c@theorem
\makeatother

\newcommand{\PY}{Pirutka and Yagita }
\newcommand{\transpose}[1]{{^t{#1}}}
\newcommand{\Gone}{G_1}
\newcommand{\Reduction}{\rho}
\newcommand{\Tor}{\mathrm{torsion}}

\title[Integral Tate conjecture]
{
On the integral Tate conjecture for finite fields}
\author
{
Masaki Kameko
}
\address
{
Department of Mathematical Sciences,
Shibaura Institute of Technology,
307 Minuma-ku Fukasaku, Saitama-City 337-8570, Japan
}
 \email
 {
kameko@shibaura-it.ac.jp
 }


\begin{document} 

\begin{abstract}
We give non-torsion counterexamples against the integral Tate conjecture for finite fields. 
We  extend the result due to \PY for prime numbers $2, 3, 5$
to
all prime numbers.
\end{abstract}

\maketitle

\section{Introduction}

The integral Tate conjecture is a conjecture on the subjectivity of the cycle 
map in algebraic geometry. 
In \cite{pirutka-yagita-2014}, \PY  proved 
Theorem~\ref{theorem:main} below for $\ell=2, 3, 5$, 
which gives a non-torsion counterexample for the integral Tate conjecture over a finite field.
We prove Theorem~\ref{theorem:main}
for all prime numbers $\ell$.
%
%
\begin{theorem} 
\label{theorem:main} 
Let $\ell$ be a prime number. There exists a smooth and projective variety 
$X$ over a finite field $k$ whose characteristic differs from $\ell$ such that
the cycle map
\[
CH^2(X_{\bar{k}}) \otimes \mathbb{Z}_{\ell} \to 
\bigcup_{U} H^4_{{e}t}(X_{\bar{k}}, \mathbb{Z}_{\ell}(2))^U/\Tor
\]
is not surjective, where $\bar{k}$ is the algebraic closure of $k$ and $U$ ranges over open subgroups 
of $\mathrm{Gal}(\bar{k}/k)$,
\end{theorem}

 We refer the reader to Colliot-Th\'{e}l\`{e}ns and Szamuely
 \cite{colliot-thelene-szamuely-2010}, 
 \PY \cite{pirutka-yagita-2014} for the details of the integral Tate conjecture.
A counterexample against the integral Tate conjecture
was provided by Atiyah and Hirzebruch in \cite{atiyah-hirzebruch-1962}
as a counterexample against the integral Hodge conjecture. 
They used torsion elements in the cohomology to give a counterexample.
In \cite{totaro-1997}, \cite{totaro-1999}, Totaro used Chow rings of 
classifying spaces of linear 
algebraic groups to study cycle maps. 
In \cite{pirutka-yagita-2014}, 
using 
\PY used the cohomology of classifying spaces
of exceptional Lie groups
to
prove Theorem~\ref{theorem:main} 
for $\ell=2$, $3$, $5$. In this paper, we reinforce the topological side of their paper 
\cite{pirutka-yagita-2014}
to extend their results to all prime numbers $\ell$.

Let $G$ be a reductive complex linear algebraic group 
or its maximal compact subgroup. 
Since the homotopy type of these topological groups are the same, 
and since we deal with the ordinary cohomology of their classifying spaces, 
we do not need to make clear distinction between a reductive complex linear 
algebraic group and a  connected Lie group.
We denote by $H^i(BG;\mathbb{Z})$ the $i$-th ordinary 
integral cohomology of the topological space $BG$,
its quotient group with respect to its torsion subgroup by 
$H^i(BG;\mathbb{Z})/\Tor$, and
the  $i$-th mod $\ell$ ordinary cohomology of $BG$
 by $H^i(BG;\mathbb{Z}/\ell)$, respectively.
We write  
\[
\Reduction:H^i(BG;\mathbb{Z})\to H^i(BG;\mathbb{Z}/\ell)
\]
for the mod $\ell$ reduction.
We denote by $Q_i$ the $i$-th Milnor operation
on the mod $\ell$ ordinary cohomology.
In particular, $Q_0$ is the Bockstein operation 
on the mod $\ell$ cohomology.
 The proposition below is nothing but Proposition~{1.4} in \cite{pirutka-yagita-2014}.
 It provides the bridge between algebraic geometry and algebraic topology.
%
%
\begin{proposition}[\PY]
\label{proposition:py}
Suppose that there exists a non-zero element $u_4
\in H^{4}(BG;\mathbb{Z})/\Tor$ such that 
$Q_1 \Reduction({u}_4)\not=0$ in $H^{2\ell+3}(BG;\mathbb{Z}/\ell)$,
then there exists 
a smooth and projective variety $X$ over a finite field 
$k$ whose characteristic is prime to $\ell$
such that the cycle map 
\[
CH^2(X_{\bar{k}}) \otimes \mathbb{Z}_{\ell} \to 
\bigcup_{U} H^4_{{e}t}(X_{\bar{k}}, \mathbb{Z}_{\ell}(2))^U/\Tor
\]
is not surjective.
\end{proposition}

With Proposition~\ref{proposition:py}, using 
Proposition~\ref{proposition:exceptional} below, 
as topological input, \PY  proved
Theorem~\ref{theorem:main} for $\ell=2, 3, 5$.
They used the mod $\ell$ cohomology of elementary abelian $\ell$-group of rank $3$ 
of these exceptional Lie groups in their proof. We 
denote the elementary abelian $\ell$-group  of rank $3$ by $A_3$.
For an odd prime number $\ell$, the ordinary mod $\ell$ cohomology 
$H^{*}(BA_3;\mathbb{Z}/\ell)$ is 
a polynomial tensor exterior algebra 
$\mathbb{Z}/\ell[x_2, y_2, z_2]\otimes 
\Lambda(x_1, y_1, z_1)$ 
where $x_1, y_1, z_1$ are degree $1$ elements corresponding to the generators of 
$A_3$ and 
$x_2=Q_0x_1, y_2=Q_0y_1, z_2=Q_0z_1$.
For $\ell=2$, the mod $2$ cohomology of $BA_3$ is a polynomial algebra 
$\mathbb{Z}/2[x_1,y_1, z_1]$.
In particular, we have $Q_1 Q_0(x_1y_1z_1)\not = 0$ in 
$H^{2\ell+3}(BA_3;\mathbb{Z}/\ell)$ 
for all prime numbers $\ell$.
%
%
\begin{proposition}\
\label{proposition:exceptional} 
For $(G, \ell)=(G_2, 2)$, $(F_4, 3)$, $(E_8, 5)$, 
there exist an elementary $\ell$-subgroup $A_3$ of $G$
of rank $3$ and 
a non-zero element $u_4\in H^4(BG;\mathbb{Z})/\Tor$ such that
$Q_1\iota^*(\Reduction(u_4))=Q_1Q_0(x_1y_1z_1)\not=0$.
\end{proposition}

In this paper, we replace Proposition~\ref{proposition:exceptional}  
by the following proposition
to obtain the proof of  Theorem~\ref{theorem:main} for all prime numbers $\ell$.
Let 
\begin{align*}
\Gone&=SU(\ell)\times SU(\ell)/\langle \Delta(\xi)\rangle,
\end{align*}
 where 
$\langle \Delta(\xi)\rangle$ is a subgroup of the center of $SU(\ell)\times SU(\ell)$.
We give an explicit description of its elementary abelian $\ell$-subgroup $A_3$ and 
$\langle \Delta(\xi) \rangle$ in \S3.
We denote by $\iota:A_3\to \Gone$ the inclusion of $A_3$ into $\Gone$.
%
%
\begin{proposition}\label{proposition:projective}
For a prime number $\ell$,
the $4$-th integral cohomology of $BG_1$ is isomorphic to 
$\mathbb{Z}\oplus\mathbb{Z}$ and the mod $\ell$ reduction
\[
\Reduction:H^4(BG_1;\mathbb{Z}) \to H^4(BG_1;\mathbb{Z}/\ell)
\]
 is an epimorphism.
Moreover,  there exists
a non-zero element $u_4\in H^4(B\Gone;\mathbb{Z})$ such that
$Q_1\iota^*(\Reduction(u_4))=Q_1Q_0(x_1y_1z_1)\not =0$.
\end{proposition}

This paper is organized as follows:
From \S2 to \S5, we assume that $\ell$ is an odd prime number.
In \S2, as preliminaries, 
we describe the non-total elementary abelian $\ell$-subgroup $A_2$ of 
the projective unitary group $PU(\ell)$, that is, the quotient group of the special unitary group $SU(\ell)$ by its center
$\mathbb{Z}/\ell$, and the Weyl group of $A_2$.
The Weyl group of an elementary abelian $\ell$-subgroup $A$  of a group $G$ is 
$N(A)/C(A)$, where $N(A)$ is 
the normalizer subgroup of $A$ in $G$ and $C(A)$
is the centralizer subgroup of $A$ in $G$.
In \S3, we define the elementary abelian $\ell$-subgroup $A_3$ above and 
a subgroup $W$ of the Weyl group of $A_3$. Then, 
we compute the set of invariants $H^{4}(BA_3;\mathbb{Z}/\ell)^{W}$.
Since $\Gone$ is a connected Lie group, the inner automorphisms act trivially on 
the cohomology of $B\Gone$. Therefore, the induced homomorphism 
\[
\iota^*:H^*(B\Gone;\mathbb{Z}/\ell)\to H^* (BA_3;\mathbb{Z}/\ell)
\]
factors through the ring of invariants $H^*(BA_3;\mathbb{Z}/\ell)^W$.
We also show that this ring of invariants contains 
$Q_0(x_1y_1z_1)$ for an odd prime number $\ell$  and that
$Q_1 Q_0(x_1y_1z_1)$ is  non-zero in $H^{*}(BA_3;\mathbb{Z}/\ell)$. 
In \S4, again, as preliminaries, 
we recall the mod $\ell$ cohomology of the classifying space of 
$PU(\ell)$ up to degree $6$.
In \S5, 
by computing Leray-Serre spectral sequences for the mod $\ell$ cohomology 
of $BG_1$ and $BA_3$ and the induced homomorphism between them, we prove 
Proposition~\ref{proposition:projective}.
In \S6, we deal with the case $\ell=2$ to complete the proof of 
Proposition~\ref{proposition:projective}.

Throughout the rest of this paper, for elements $g, h$ in a group, we denote $h^{-1}g h$
by $g^h$. We also write $[g,h]$ for the commutator 
$g^{-1} h^{-1} g h$. For elements $g_0, g_1, \dots$ in a group, 
we denote by $\langle g_0,g_1, \dots \rangle$ 
the subgroup generated by $g_0, g_1, \dots$.
Also, for a ring $R$ and for a finite set $\{ m_0, \dots, m_r\}$, 
we denote by $R\{ m_0, \dots, m_r\}$ the free $R$-module spanned by 
$\{m_0, \dots, m_r\}$.

After the author sent a preliminary version of this paper to Nobuaki
Yagita, Yagita informed 
the author that Totaro used the group $
(SL(\ell)\times SL(\ell)/\mathbb{Z}/\ell) \times \mathbb{Z}/\ell$
to study the geometric and topological filtration of the complex representation ring
in his  quite recently published book \cite[\S15]{totaro-2014}. 
In the same time, Yagita encouraged the author to publish this paper. 
The author would like to thank Yagita for his kind encouragement.
The author is partially supported 
by the Japan Society for the Promotion of Science, 
Grant-in-Aid for Scientific Research (C) 25400097.


\section{The elementary abelian $\ell$-subgroup $A_2$ of $PU(\ell)$}

In this section, we recall the non-total maximal elementary abelian 
$\ell$-group $A_2$ in $PU(\ell)$ and the Weyl group of $A_2$.

First, we define the elementary abelian 
$\ell$-group $A_2$. Let $\xi=\exp(2\pi i/\ell)\in \mathbb{C}$ 
and let $I$ be the identity matrix in $SU(\ell)$.
By abuse of notation, we write $\xi$ for $\xi I$.
We define unitary matrixes $\alpha$, $\beta$ with determinant $1$ by
\begin{align*}
\alpha&=(\delta_{ij} \xi^i)=\mathrm{diag}(\xi^1, \xi^2, \cdots, \xi^{\ell}),\\
\beta&=(\delta_{i,j+1})
\end{align*}
where $\delta_{ij}=1$ if $i\equiv j \mod \ell$ 
and $\delta_{ij}=0$ if $i\not \equiv j \mod \ell$. 
Indeed, $\alpha^{-1}={\transpose{\bar{\alpha}}}
=\mathrm{diag}(\xi^{-1}, \xi^{-2}, \dots, \xi^{-\ell})$, 
$\beta^{-1}=\transpose{\bar{\beta}}=(\delta_{i,j-1})$. 
By direct computation, we obtain
\[
[\alpha,\beta] = \xi.
\]
Therefore, the subgroup
\[
A_2=\langle \alpha, \beta, \xi\rangle /\langle \xi \rangle
\]
 of 
$PU(\ell)=SU(\ell)/\langle \xi \rangle$ generated by $\alpha, \beta$ is
an elementary abelian $\ell$-subgroup of $PU(\ell)$.
We denote by $\iota:A_2 \to PU(\ell)$ the inclusion map.

Next, we recall  inner automorphisms of $SU(\ell)$ or $PU(\ell)$ 
which preserve $A_2$ in order to study the image of the induced homomorphism 
\[
\iota^*:H^{*}(BPU(\ell);\mathbb{Z}/\ell)
\to H^{*}(BA_2;\mathbb{Z}/\ell).
\]
It is well-known that the Weyl group of $A_2$ in $PU(\ell)$ is 
the finite special linear group $SL_2(\mathbb{Z}/\ell)$. 
Nevertheless, we give explicit matrix 
generators $\sigma, \tau$ for the Weyl group $SL_2(\mathbb{Z}/\ell)$ 
hoping it might be useful some day.

In order to define unitary matrixes $\sigma, \tau$ with determinant $1$, we consider the following sequence
of integers:
\[
a_0=0, \quad a_{i}=i+a_{i-1} \quad \mbox{$i\geq 1$}.
\]
We also use the following lemmas to define $\sigma$, $\tau$.
%
%
\begin{lemma}
\label{lemma:lemma-1}
It holds that
\[
\ell^{-1} \sum_{k=1}^{\ell} \xi^{km}=\delta_{n,0}.
\]
\end{lemma}
\begin{proof}
If $m\equiv 0 \mod \ell$, then $\xi^m=1$ in $\mathbb{C}$.
Hence, we have $\displaystyle \sum_{k=1}^{\ell} \xi^{km}=\ell$
for $m \equiv 0 \mod \ell$.
If $m\not\equiv 0 \mod \ell$, then $1-\xi^m\not=0$ in $\mathbb{C}$. Since
\[
(1-\xi^m)
\left( \sum_{k=1}^\ell \xi^{km} \right)
=\xi^m-\xi^{\ell m+m}=0,
\]
we have $\displaystyle \sum_{k=1}^{\ell} \xi^{km}=0$ for $m \not \equiv 0 \mod \ell$.
\end{proof}
%
%
\begin{lemma}
\label{lemma:lemma-2}
It holds that
\[
a_{j+k}-a_{i+k}\equiv
k(j-i) + (a_j-a_i).
\]
\end{lemma}
\begin{proof}
Inductively, we have
\begin{align*}
a_{j+k}-a_{i+k}
&
=
(j+k+a_{j+k-1})-(i+k+a_{i+k-1})
\\
& =  (j-i)+(a_{j+k-1}-a_{i+k-1}) \\
& \;\; \vdots \\
& =
k(j-i)+(a_j-a_i). \qedhere
\end{align*} 
\end{proof}

Now, we define unitary matrixes $\sigma, \tau$ with determinant $1$ in $SU(\ell)$.
Let us define matrixes $S$, $T$  by
\begin{align*}
S&=\mathrm{diag}(\xi^{a_1}, \xi^{a_2}, \dots, \xi^{a_{\ell}}), \\
T&=(\xi^{a_{i+j}}).
\end{align*}
It is clear that 
$S^{-1}=\transpose{\bar{S}}=\mathrm{diag}(\xi^{-a_1}, \xi^{-a_2}, \dots, \xi^{a_{\ell}})$.
The $(i,j)$-entry of $\transpose{\bar{T}}T$ is
\[
\sum_{m,n=1}^{\ell} \xi^{-a_{i+m}} \delta_{mn} \xi^{a_{j+n}}.
\]
Put $k=m=n$. Then, the $(i,j)$-entry
of 
$\transpose{\bar{T}}T$ is
\[
\sum_{k=1}^{\ell} \xi^{-a_{i+k}}  \xi^{a_{j+k}}
=\sum_{k=1} \xi^{a_{j+k}-a_{i+k}}=\xi^{a_j-a_i} 
\sum_{k=1}^\ell \xi^{k(j-i)}=\delta_{ij} \ell \xi^{a_j-a_i}=\delta_{ij} \ell
\]
So, we have 
$\transpose{\bar{T}}T=\ell I$.
The determinants of $S$, $(\sqrt{\ell})^{-1} T$ are in 
$\{ z\in \mathbb{C} \;|\; |z|=1 \}$.
Hence,  there exist $\theta_0, \theta_1$ such that
$\det S=\exp(i \theta_0)$, $\det (\sqrt{\ell})^{-1} T=\exp(i \theta_1)$. 
Put $\mu_0=\exp(i\theta_0/\ell)$, 
$\mu_1=\exp(i\theta_1/\ell)$. 
We define $\sigma, \tau$ by $\mu_0^{-1} S$, 
$ (\mu_1\sqrt{\ell})^{-1} T$, respectively, so that
$\tau, \sigma$ are unitary matrixes with determinant $1$.

We end this section with the following proposition on the inner automorphisms
defined by $\sigma$ and $\tau$.

%
%
\begin{proposition} \label{proposition:auto}
We have
\[
\sigma^{-1} \alpha\sigma=\alpha, \; \sigma^{-1}\beta \sigma=\alpha^{-1}\beta, \; 
\tau^{-1} \alpha\tau=\alpha^{-1}\beta, \; \tau^{-1}\beta\tau=\beta^{-1}.
\]
\end{proposition}

\begin{proof}
The first equality follows from the fact that both $\alpha, \sigma$ are diagonal matrixes.
Next, we consider the second equality.
The $(i,j)$-entry of $\sigma^{-1}\beta\sigma=\transpose{\bar{S}}\beta S$ is 
\[
\sum_{m,n=1}^{\ell} \delta_{im} \xi^{-a_i} \delta_{m,n+1} \delta_{nj} \xi^{a_j}.
\]
If $\delta_{i,m}=\delta_{m,n+1}=\delta_{nj}=1$, 
then
$i\equiv m\equiv n+1\equiv j+1 \mod \ell$.
So, the above entry is
\[
\delta_{i,j+1} \xi^{a_j-a_i}=\xi^{-i} \delta_{i,j+1},
\]
which is the $(i,j)$-entry of $\alpha^{-1} \beta$.
Next, we prove the third equality.
The $(i,j)$-entry of $\tau^{-1}\alpha \tau=(\ell^{-1}) \transpose{\bar{T}} \alpha T$ is 
\[
\ell^{-1} 
\sum_{m,n=1}^{\ell} 
\xi^{-a_{i+m}} \delta_{mn} \xi^m \xi^{a_{j+n}.}
\]
Put $m=n=k$. Then the $(i,j)$-entry above is 
\[
\ell^{-1} \sum_{k=1}^{\ell} \xi^{-a_{j+k}} \xi^k \xi^{a_{j+k}}=
\ell^{-1} \sum_{k=1}^{\ell} \xi^{a_{j+k}-a_{i+k}+k}=\xi^{a_j-a_i}{\ell}^{-1}
\sum_{k=1}^{\ell} \xi^{k(j-i+1)}=\delta_{i,j+1} \xi^{a_j-a_i}.
\]
So, as in the proof of the second equality, 
it is equal to the $(i,j)$-entry of $\alpha^{-1} \beta$.

Finally we prove the fourth equality. The $(i,j)$-entry of $\tau^{-1} \beta \tau$
is
\[
\ell^{-1}\sum_{m,n=1}^{\ell} \xi^{-a_{i+m}} \delta_{m,n+1} \xi^{a_{j+n}}.
\]
Put $m=n+1=k$. Then 
the $(i,j)$-entry above is equal to
\[
\ell^{-1} \sum_{k=1}^{\ell} \xi^{a_{j+k-1}-a_{i+k}}
={\xi^{a_{j-1}-a_i}}{\ell}^{-1}
\sum_{k=1}^{\ell} 
\xi^{k(j-1-i)}=\delta_{i,j-1} \xi^{a_{j-1}-a_i}=\delta_{i,j-1}.
\]
Hence, we have $\tau^{-1}\beta \tau=\beta^{-1}$.
\end{proof}

Thus, the matrix representing the inner automorphisms defined by  $\sigma$, $\tau$ are given by
\[
(\alpha, \beta)^{\sigma}
=(\alpha, \beta) \begin{pmatrix} 1  & -1   \\   0  &  1 \end{pmatrix}, \quad
(\alpha, \beta)^{\tau}
=(\alpha, \beta) \begin{pmatrix} -1  & 0   \\   1  &  -1 \end{pmatrix},
\]
respectively. It is clear that \[
\begin{pmatrix} 1  & -1   \\   0  &  1 \end{pmatrix}^{\ell-1}
=\begin{pmatrix} 1  & 1   \\   0  &  1 \end{pmatrix}, 
\quad 
\begin{pmatrix} -1  & 0   \\   1  &  -1 \end{pmatrix}^{\ell-1}
=\begin{pmatrix} 1  & 0   \\   1  &  1 \end{pmatrix}
\]
and 
these matrixes generate the special linear group $SL_2(\mathbb{Z}/\ell)$.


\section{The elementary abelian $\ell$-subgroup $A_3$ of $G_1$}

For an odd prime number $\ell$, we define the connected Lie group $G_1$, its elementary abelian $\ell$-subgroup $A_3$ and the subgroup $W$ of the Weyl group of $A_3$, we mentioned in the 
introduction.
We denote by $\iota:A_3\to G_1$ the inclusion of $A_3$ and
we consider the induced homomorphism 
\[
\iota^*:H^{*}(BG_1;\mathbb{Z}/\ell) \to H^{*}(BA_3;\mathbb{Z}/\ell)^W
\]
As in the previous section, let $I$ be the identity matrix in $SU(\ell)$, 
 $\xi=\exp(2\pi i/\ell)$ in $\mathbb{C}$ and by abuse of notation, we 
denote by $\xi$ the matrix $\xi I$.
We define the connected Lie group $G_1$ by 
\begin{align*}
\Gone&=SU(\ell)\times SU(\ell)/\langle \Delta(\xi)\rangle, 
\end{align*}
where 
\[
\Delta:SU(\ell) \to SU(\ell) \times SU(\ell)
\]
is the diagonal map sending $Y \in SU(\ell)$ to $
\displaystyle \left(\begin{array}{cc} Y & 0 \\ 0 & Y \end{array}
\right) \in SU(\ell)\times SU(\ell)$.
We also consider a map
\[
\Gamma:SU(\ell)\to SU(\ell)\times SU(\ell)
\]
 sending $Y \in SU(\ell)$ to $\displaystyle
\left( \begin{array}{cc} I & 0 \\  0 & Y
\end{array}
\right) \in SU(\ell)\times SU(\ell)$.
We define $A_3$ to be
\[
A_3=\langle 
\Delta(\alpha), \Delta(\beta), \Delta(\xi), \Gamma(\xi) \rangle /\langle 
\Delta(\xi) \rangle.
\]
It is easy to see that
\begin{align*}
[\Delta(\alpha), \Delta(\beta)]&=\Delta(\xi) , 
\\
[\Gamma(\xi), \Delta(\alpha)]&=\Delta(I), \\
[\Gamma(\xi), \Delta(\beta)]&=\Delta(I)
\end{align*}
 in $SU(\ell)\times SU(\ell)$.
Therefore, 
$A_3$ is an elementary abelian $\ell$-subgroup of $\Gone$.

Next, we consider inner automorphisms of $G_1$ preserving $A_3$.
By Proposition~\ref{proposition:auto}, matrixes corresponding to the inner automorphisms defined by $\Delta(\sigma)$, 
$\Delta(\tau)$ 
are given as follows:
\[
(\Delta(\alpha), \Delta(\beta), \Gamma(\xi))^{\Delta(\sigma)}=
(\Delta(\alpha), \Delta(\beta), \Gamma(\xi))
\begin{pmatrix}
1 & -1 & 0 \\ 
0 & 1 & 0 \\
0 & 0 & 1  \end{pmatrix}
,
\]
\[
(\Delta(\alpha), \Delta(\beta), \Gamma(\xi))^{\Delta(\tau)}=
(\Delta(\alpha), \Delta(\beta), \Gamma(\xi))\begin{pmatrix}
-1 & 0 & 0 \\ 
1 & -1 & 0 \\
0 & 0 & 1 \end{pmatrix}.
\]
By direct computation, it is also easy to verify that 
 \begin{align*}
\Delta(\alpha)^{\Gamma(\beta)}&= \Gamma(\xi) \Delta(\alpha), \\
\Delta(\beta)^{\Gamma(\beta)}&=\Delta(\beta), \\
\Gamma(\xi)^{\Gamma(\beta)}&=\Gamma(\xi).
\end{align*}
So, the matrix corresponding to the inner automorphism defined by $\Gamma(\beta)$ is 
given by
\[
(\Delta(\alpha), \Delta(\beta), \Gamma(\xi))^{\Gamma(\xi)} =
(\Delta(\alpha), \Delta(\beta), \Gamma(\xi))
\begin{pmatrix} 1 & 0 & 0 \\ 0 & 1 & 0 \\ 1 & 0 & 1 \end{pmatrix}.
\]
Thus, the inner automorphisms defined by $\Delta(\sigma)^{\ell-1}$, 
$\Delta(\tau)^{\ell-1}$, 
 $\Gamma(\beta)$ correspond to matrixes
\[
\begin{pmatrix}
1 & 1 & 0 \\ 
0 & 1 & 0 \\
0 & 0 & 1 \end{pmatrix}, \;
\begin{pmatrix}
1 & 0 & 0 \\ 
1 & 1 & 0 \\
0 & 0 &  1 \end{pmatrix}, \;
\begin{pmatrix}
1 & 0 & 0 \\ 
0 & 1 & 0 \\
1 & 0 &  1 \end{pmatrix},
\]
and  they generate the following subgroup $W$ of the Weyl group.
\[
W=\left. \left\{ \begin{pmatrix} a & b & 0 \\ c & d & 0 \\ * &  0 & 1 \end{pmatrix}
 \; \right |\; ad-bc=1
\right\}.
\]

Finally, we compute the set of invariants  $H^4(BA_3;\mathbb{Z}/\ell)^W$
by direct calculation. 
%
%
\begin{proposition}
\label{proposition:odd-invariants}
For an odd prime number $\ell$, we have
\[
H^4(BA_3;\mathbb{Z}/\ell)^W=\mathbb{Z}/\ell\{ Q_0(x_1y_1z_1)\}.
\]
Moreover, we have
\[
Q_1(Q_0(x_1y_1z_1))=-x_2^{\ell} y_2 z_1+x_2^\ell y_1z_2+ 
x_2y_2^\ell z_1-x_2y_1z_2-x_1y_2^\ell z_2+x_1 y_2 z_2^\ell \not=0
.
\]
\end{proposition}
\begin{proof}
Let $x_1, y_1, z_1$ be generators of $H^1(BA_3;\mathbb{Z}/\ell)$
corresponding to $\Delta(\alpha), \Delta(\beta), \Gamma(\xi)$ in $A_3$, respectively.
Let  
\[
W_0 = \left. \left\{ \begin{pmatrix} a & b & 0 \\ c & d & 0 \\
0 & 0 & 1 \end{pmatrix}
\;
\right|  
\;ad-bc =1 \right\}.
\]
For an odd prime number $\ell$, the ring of invariants of 
the polynomial tensor exterior algebra 
$\mathbb{Z}/\ell[x_2, y_2]\otimes \Lambda(x_1,y_1)$
with respect to the action of the finite special linear group 
$SL_2(\mathbb{Z}/\ell)$  is known as 
Dickson-Mui invariants. We refer the reader to 
Mui \cite{mui-1975} or Kameko and Mimura \cite{kameko-mimura-2007} 
for the details of the Dickson-Mui invariants. 
It is equal to 
\[
\mathbb{Z}/\ell[u_{2\ell+2}, u_{2\ell^2-2\ell}]
\{ 1, x_1y_1, Q_0(x_1y_1), Q_1(x_1y_1)\},
\]
where 
\begin{align*}
u_{2\ell+2}&=Q_1Q_0(x_1y_1)=x_2y_2^\ell-x_2^\ell y_2, 
\\
u_{2\ell^2-2\ell}&=Q_2 Q_0(x_1y_1)/Q_1Q_0(x_1y_1)=\sum_{k=0}^{\ell}
x_2^{k(\ell-1)} y_2^{(\ell-k)(\ell-1)},
\\
Q_0(x_1y_1)& =x_2y_1-x_1y_2,\\
Q_1(x_1y_1)& =x_2^{\ell}y_1-x_1y_2^{\ell}.
\end{align*}
Therefore, the set of invariants 
$H^4(BA_3;\mathbb{Z}/\ell)^{W_0}$ is spanned by
$z_2^2$, $x_1y_1 z_2$, $(x_2y_1-x_1y_2)z_1$ as a $\mathbb{Z}/\ell$ vector space.
Let $f:A_3 \to A_3$ be the inner automorphism corresponding to 
\[
\begin{pmatrix} 1 & 0 & 0 \\ 0 & 1 & 0 \\ 1 & 0 & 1
\end{pmatrix}.
\]
Then, the induced homomorphism $f^*:H^{*}(BA_3;\mathbb{Z}/\ell) \to 
H^{*}(BA_3;\mathbb{Z}/\ell)$ maps $x_i, y_i, z_i$ to
$x_i, y_i, x_i+z_i$, respectively. Thus, we have
\begin{align*}
(1-f^*)(z_2^2)&=-x_2^2, \\
(1-f^*)(x_1y_1z_2)&= x_1x_2y_1, \\
(1-f^*)((x_2y_1-x_1y_2)z_1)&=-x_1x_2y_1.
\end{align*}
Hence, the kernel of 
\[
(1-f^*):H^{4}(BA_3;\mathbb{Z}/\ell)^{W_0}\to
H^{4}(BA_3;\mathbb{Z}/\ell)
\]
is spanned by 
\[
x_1y_1z_2+(x_2y_1-x_1y_2)z_1=Q_0(x_1y_1z_1). 
\]
Since $W_0$ and $f$ generated the subgroup $W$ of the Weyl group, the kernel of the induced homomorphism $1-f^*$ above is the ring of invariants $H^{*}(BA_3;\mathbb{Z}/\ell)^W$. 
This completes the proof of Proposition~\ref{proposition:odd-invariants}.
\end{proof}


\section{The mod $\ell$ cohomology of $BPU(\ell)$ up to degree $6$}

We recall the following Proposition~\ref{proposition:pu}
 on the mod $\ell$ cohomology of $BPU(\ell)$. 
The mod $\ell$ cohomology of the classifying space $BPU(\ell)$ was computed by Kono, 
Mimura and Shimada in \cite{kono-mimura-shimada-1975} for $\ell=3$.
For all odd prime numbers, computation was done 
by Kameko and Yagita in \cite{kameko-yagita-2008} and 
by Vistoli in \cite{vistoli-2007}, independently. 
However, what we need in this paper is the computation 
up to degree $6$ only. So, we compute the mod $\ell$ cohomology of $BPU(\ell)$ up to degree $6$
instead of referring the reader to \cite{kameko-yagita-2008} or \cite{vistoli-2007}.

We say the spectral sequence $E_r^{p,q}$ collapses at the 
$E_m$-level up to degree $n$ if 
$E_m^{p,q}=E_{m+1}^{p,q}=\cdots=E_{\infty}^{p,q}$ for $p+q\leq n$.
We say $M$ is a free $R$-module up to degree $n$ if there exists 
a free $R$-module $M_0$ and an $R$-module homomorphism
$f:M_0\to M$ such that $f:M_0^{p,q}\to M^{p,q}$ is an isomorphism for all $p+q\leq n$.

%
%
\begin{proposition}
\label{proposition:pu}
For an odd prime number $\ell$, 
$H^{*}(BPU(\ell);\mathbb{Z}/\ell)$ is spanned by $1, v_2, v_2^2, v_2^3, v_3$ up to 
degree $6$ as a graded $\mathbb{Z}/\ell$-module, 
where $v_2, v_3$ are of degree $2$, $3$, respectively.
In particular, $v_2v_3=0$.
Moreover, the induced homomorphism 
\[
\iota^*:H^2(BPU(\ell);\mathbb{Z}/\ell) \to 
H^2(BA_2;\mathbb{Z}/\ell)^{SL_2(\mathbb{Z}/\ell)}
\] 
is an isomorphism.
\end{proposition}

\begin{proof}
Let us consider the Leray-Serre spectral sequence 
associated with the fibre sequence
\[
BU(\ell) \stackrel{j}{\longrightarrow} 
BPU(\ell)\stackrel{\varphi}{\longrightarrow} 
K(\mathbb{Z},3).
\]
First, we describe its $E_2$-term 
\[
H^{*}(K(\mathbb{Z}, 3);\mathbb{Z}/\ell) \otimes H^{*}(BU(\ell);\mathbb{Z}/\ell).
\]
We denote by $u_3, u_{2\ell+1}$ the algebra generators 
of the mod $\ell$ cohomology of the 
Eilenberg-MacLane space $K(\mathbb{Z}, 3)$ up to degree $2\ell+1$, so that 
\[
H^{*}(K(\mathbb{Z}, 3);\mathbb{Z}/\ell)=\mathbb{Z}/\ell\{1, u_3, u_{2\ell+1}\}
\]
as a graded vectors space.
We denote  algebra generators 
of the mod $\ell$ cohomology of $BU(\ell)$ by $y_2, \dots, y_{2\ell}$
where $\deg y_i=i$.
The mod $\ell$ cohomology of $BU(\ell)$ is a polynomial algebra 
\[
\mathbb{Z}/\ell[y_2, \dots, y_{2\ell}].
\]
The $E_2$-term is, up to degree $7$, 
spanned by 
\[
1, y_2, y_2^2, y_2^3, y_4, y_2y_4, y_6;  u_3, y_2u_3, y_2^2 u_3, y_4u_3,
\]
for $\ell>3$, and 
by 
\[
1, y_2, y_2^2, y_2^3, y_4, y_2y_4, y_6;  u_3, y_2u_3, y_2^2 u_3, y_4u_3; u_7
\]
for $\ell=3$.

Next, we consider differentials.
The image of the cohomology suspension 
\[
\sigma:H^{*}(X;\mathbb{Z}/\ell) \to H^{*}(\Omega X;\mathbb{Z}/\ell)
\]
is contained in the set of primitive elements.
For $X=BU(\ell)$,   
\[
H^{*}(\Omega X;\mathbb{Z}/\ell)=H^{*}(U(\ell);\mathbb{Z}/\ell)=
\Lambda(\sigma(y_2), \dots, \sigma(y_{2\ell})).
\]
On the other hand, 
\[
H^{*}(PU(\ell);\mathbb{Z}/\ell)
=\mathbb{Z}/\ell[x_2]/(x_2^\ell) \otimes 
\Lambda(x_1, x_3, \cdots, x_{2i-1}, \dots, x_{2\ell-1})
\]
the subspace spanned by the primitive elements are spanned by $1, x_1,x_2$.
See Baum and Browder \cite{baum-browder-1965}
 for the details of the mod $\ell$ cohomology of $PU(\ell)$.
 So, the cohomology suspension maps  any elements of degree greater than 
 $3$ in the mod $\ell$  cohomology of $BPU(\ell)$ to zero.
Consider elements
$y_4$, 
$y_6+\alpha y_2 y_4$, 
in $H^{*}(BU(\ell);\mathbb{Z}/\ell)=E_2^{0,*}$,
where $\alpha \in \mathbb{Z}/\ell$.
Then, 
$\sigma(y_4) \not=0$ and 
$\sigma(y_6+\alpha y_2 y_4)=\sigma(y_6) \not =0$ in 
$H^{*}(U(\ell);\mathbb{Z}/\ell)$.
Hence, these elements $y_4$, 
$y_6+\alpha y_2 y_4$
in $H^{*}(BU(\ell);\mathbb{Z}/\ell)$ are  not in the image of 
the induced homomorphism
\[
j^*:H^{*}(BPU(\ell);\mathbb{Z}/\ell) \to H^{*}(BU(\ell);\mathbb{Z}/\ell).
\]
Therefore, in the Leray-Serre spectral sequence, 
the elements $y_4$, $y_6+\alpha y_2y_4$ in $E_2^{0,*}$ must support nontrivial differentials.
Since $y_4$ supports non-trivial differential, it must be 
$d_3(y_4)=\alpha' y_2u_3$ for some $\alpha'\not=0$ in $\mathbb{Z}/\ell$.
Suppose that $d_3(y_6)=  \beta' y_4 u_3+ \gamma' y_2^2 u_3$.
If $\beta'\not=0$, 
the image of the differential $d_3$ is spanned by 
$y_2u_3$ up to degree $6$ and 
the kernel of $d_3$ is spanned by 
$1, y_2, y_2^2, y_2^3, u_3, y_2u_3$ up to degree $6$.
Hence, the $E_4$-term is spanned by $1, y_2, y_2^2, u_3$ up to degree $6$.
It is clear that for dimensional reasons, these elements are permanent cocycles, 
so that $E_3^{p,q}=E_\infty^{p,q}$ for $p+q\leq 6$.
If $\beta'=0$, then
the image of $d_3$ is spanned by $y_2u_3$ and 
the kernel of $d_3$ is spanned by
$1, y_2, y_2^2, y_2^3, u_3, y_2u_3, y_6-(\gamma'/\alpha')y_2y_4$
up to degree $6$.
Hence, the $E_4$-term is spanned by 
$1, y_2, y_2^2, y_2^3, u_3, y_6-(\gamma'/\alpha')y_2y_4$ up
to degree $6$.
However, $y_6-(\gamma'/\alpha')y_2y_4$ does not survive to the 
$E_\infty$-term and $1, y_2, y_2^2, y_2^3, u_3$ are permanent cocycles, and so 
the $E_\infty$-term is spanned by $1, y_2, y_2^2, y_2^3, u_3$, up to degree $6$, anyway.
(The fact is that the above $\beta'$ is always non-zero although
 we do not give a proof here.)

Finally, we consider the induced homomorphism
\[
\iota^*: H^2(BPU(\ell);\mathbb{Z}/\ell) \to 
H^2(BA_2;\mathbb{Z}/\ell)^{SL_2(\mathbb{Z}/\ell)}.
\] 
Consider the commutative diagram of groups:
\[
\begin{diagram}
\node{\ell_{+}^{1+2}} \arrow{e,t}{\iota}\arrow{s,l}{\pi}
 \node{SU(\ell)} \arrow{s,r}{\pi} \\
\node{A_2} \arrow{e,t}{\iota} \node{PU(\ell),}
\end{diagram}
\]
where $\pi$ is the obvious projection and $\ell^{1+2}_{+}$ is the 
subgroup of $SU(\ell)$ generated by $\alpha, \beta, \xi$.
As a group, $\ell^{1+2}_{+}$ is the extra-special $\ell$-group 
of order $\ell^3$ with exponent $\ell$.
The group extension 
\[
\mathbb{Z}/\ell  \longrightarrow \ell^{1+2}_{+} \stackrel{\pi}{\longrightarrow} A_2
\]
is not trivial
and 
the induced map $\varphi' \circ \iota: BA_2\to K(\mathbb{Z}/\ell, 2)$ 
is not null-homotopic, where 
$\mathbb{Z}/\ell$
is the cyclic group of order $\ell$ generated by $\xi$.
Since $\varphi' \circ \iota$ represents a nontrivial element in 
$H^2(BA_2;\mathbb{Z}/\ell)^{SL_2(\mathbb{Z}/\ell)}=\mathbb{Z}/\ell$, 
the induced homomorphism 
\[
H^2(BPU(\ell);\mathbb{Z}/\ell) \to H^2(BA_2;\mathbb{Z}/\ell)^{SL_2(\mathbb{Z}/\ell)}\]
is an isomorphism.
\end{proof}


\section{The Leray-Serre spectral sequences}

In this section, we prove Proposition~\ref{proposition:projective}.
To this end, we prove Proposition~\ref{proposition:bg1} below.
%
%
\begin{proposition}
\label{proposition:bg1}
The $4$-th mod $\ell$ cohomology of $B\Gone$ as a vector space over 
$\mathbb{Z}/\ell$ is given as follows:
\[
H^4(B\Gone;\mathbb{Z}/\ell)=\mathbb{Z}/\ell
\oplus \mathbb{Z}/\ell.
\]
Moreover, the induced homomorphism 
\[
\iota^*:H^4(B\Gone;\mathbb{Z}/\ell) \to H^4(BA_3;\mathbb{Z}/\ell)^{W}
\]
is an epimorphism.
\end{proposition}


\begin{proof}[Proof of Proposition~\ref{proposition:projective} 
modulo Proposition~\ref{proposition:bg1}]
Since the rational cohomology of $B\Gone$ is isomorphic to that of 
$B(SU(\ell)\times SU(\ell))$, 
it is a polynomial algebra 
generated by $2(\ell-1)$ elements of degree $4, 4, 6, 6, \dots, 2\ell, 2\ell$.
In particular, $H^4(B\Gone;\mathbb{Q})=\mathbb{Q}\oplus \mathbb{Q}$.
For a topological space $X$ of the homotopy type of a CW complex of finite type, 
\[
\dim_{\mathbb{Q}} H^i(X;\mathbb{Q})\leq \dim_{\mathbb{Z}/\ell} H^i(X;\mathbb{Z}/\ell).
\]
If
\[
\dim_{\mathbb{Q}} H^i(X;\mathbb{Q})= \dim_{\mathbb{Z}/\ell} H^i(X;\mathbb{Z}/\ell),
\]
the mod $\ell$ reduction 
\[
\Reduction:H^i(X;\mathbb{Z}) \to H^i(X;\mathbb{Z}/\ell)
\]
is an epimorphism and $H^i(X;\mathbb{Z})$ is torsion free.
Therefore, by Proposition~\ref{proposition:bg1},
we have that
$H^4(BG;\mathbb{Z})=\mathbb{Z}\oplus \mathbb{Z}$ and 
that the mod $\ell$ reduction 
\[
\Reduction:H^{4}(B\Gone;\mathbb{Z})
\to H^{4}(B\Gone;\mathbb{Z}/\ell)
\]
 is an epimorphism.
Therefore, by Proposition~\ref{proposition:odd-invariants},
we have the existence of non-zero element
$u_4\in H^4(B\Gone;\mathbb{Z})$ 
such that $\Reduction(u_4)=Q_0(x_1y_1z_1)$.
It implies Proposition~\ref{proposition:projective} for all odd prime numbers $\ell$.
\end{proof}


Now, we prove Proposition~\ref{proposition:bg1} above.
It is clear that $\Gone/\langle  \Gamma(\xi) \rangle=PU(\ell)\times PU(\ell)$.
We consider the following commutative diagram.
\[
\begin{diagram}
\node{A_3} \arrow{e,t}{\iota} \arrow{s,l}{\pi}
\node{\Gone}\arrow{s,r}{\pi}
\node{SU(\ell)} \arrow{w,t}{\Gamma} \arrow{s,r}{\pi} \\
\node{A_2} \arrow{e,t}{\iota}
\node{PU(\ell)\times PU(\ell)}
\node{PU(\ell)} \arrow{w,t}{\Gamma}
\end{diagram}
\]
where
the map $\iota:A_2\to PU(\ell)\times PU(\ell)$ 
is the composition of the diagonal map $$\Delta:PU(\ell)
\to PU(\ell)\times PU(\ell)$$ and  
the inclusion of $A_2$ into $PU(\ell)$.
From the above commutative diagram, we have 
fibre squares
\[
\begin{diagram}
\node{BA_3} \arrow{e,t}{\iota} \arrow{s,l}{\pi}
\node{B\Gone}\arrow{s,r}{\pi}
\node{BSU(\ell)} \arrow{w,t}{\Gamma} \arrow{s,r}{\pi}  \\
\node{BA_2} \arrow{e,t}{\iota}
\node{B(PU(\ell)\times PU(\ell))}
\node{BPU(\ell)} \arrow{w,t}{\Gamma}
\end{diagram}
\]
We consider the Leray-Serre spectral sequences associated with
vertical fibrations and the induced homomorphism between them.
For $Y=B\Gone, BA_3, BSU(\ell)$, 
we denote by $E_r^{p,q}(Y)$  the spectral sequences associated with the 
above fibrations converging to the mod $\ell$ cohomology 
$H^{*}(Y;\mathbb{Z}/\ell)$.
We also write 
$\iota^*:E_r^{p,q}(B\Gone)\to E_r^{p,q}(BA_3)$, $\Gamma^*:E_r^{p,q}(BSU(\ell))
\to E^{p,q}_r(BG)$ for  the induced homomorphisms.

We  compute the Leray-Serre spectral sequence for 
$H^{*}(B\Gone;\mathbb{Z}/\ell)$ up to degree $4$, 
starting with  the $E_2$-term up to degree $6$.

First, we describe the $E_2$-term up to degree $6$.
Identifying 
the mod $\ell$ cohomology of $B(PU(\ell)\times PU(\ell))$ with 
\[
H^{*}(BPU(\ell);\mathbb{Z}/\ell)\otimes H^{*}(BPU(\ell);\mathbb{Z}/\ell),
\]
let us consider the following algebra generators of 
the mod $\ell$ cohomology of the classifying space 
$B(PU(\ell)\times PU(\ell))$:
\[ \begin{array}{lll}
1=1\otimes 1, & \quad a_2=v_2\otimes 1-1 \otimes v_2, & \quad
a_3=v_3\otimes 1-1\otimes v_3, \\
& \quad b_2=v_2\otimes 1, & \quad 
b_3=v_3\otimes 1.
\end{array}
\] 
Then, up to degree $6$, the mod $\ell$ cohomology of $B(PU(\ell)\times PU(\ell))$ 
is  a free $\mathbb{Z}/\ell[a_2]$-module 
\[
\mathbb{Z}/\ell[a_2] \{ 1, b_2,b_2^2, b_2^3, a_3, b_3, a_3b_3\}.
\]

Next, we consider non-trivial differentials.
We denote by $z_1$, $z_2=Q_0 z_1$ the algebra generators 
of degree $1$, $2$ of the mod $\ell$ cohomology of 
the fibre $B\langle\Gamma(\xi)\rangle=B\langle \xi \rangle$ of the projection $\pi$, 
so that
\[
H^{*}(B\langle\Gamma(\xi)\rangle;\mathbb{Z}/\ell)
=\mathbb{Z}/\ell[z_2] \otimes \Lambda(z_1).
\]
By definition, $\Gamma^{*}(b_i)=0$, $\Gamma^*(a_i)=-v_i$ for $i=2,3$.
Moreover, by choosing suitable $v_2$, $v_3$,  we may assume that
$\iota^*(a_i)=0$ for $i=2,3$, $\iota^*(b_2)=x_1y_1$, 
$\iota^*(b_3)=x_2y_1-x_1y_2$.
Since $v_2, v_3$ are in the image of the induced homomorphism
\[
{\varphi'}^*: H^{*}(K(\mathbb{Z}/\ell, 2);\mathbb{Z}/\ell) \to 
H^{*}(BPU(\ell);\mathbb{Z}/\ell), 
\]
in the  spectral sequence $E_r^{*,*}(BSU(\ell))$, 
$d_2(z_1)=
\alpha_1 v_2$ and
$d_3(z_2)=\alpha_2 v_3$ for some 
$\alpha_1\not=0, \alpha_2\not=0$ in $\mathbb{Z}/\ell$.
On the other hand, since the induced homomorphism $H^{*}(BA_3;\mathbb{Z}/\ell)\to
H^{*}(B\langle \Gamma(\xi)\rangle;\mathbb{Z}/\ell)$ is an epimorphism, 
the spectral sequence $E_r^{*,*}(BA_3)$ collapses at the $E_2$-level
and $d_2(z_1)=d_3(z_2)=0$ in the spectral sequence $E_r^{*,*}(BA_3)$.
Thus, we have non-trivial differentials $d_2(z_1)=-\alpha_1 a_2$, 
$d_3(z_2)=-\alpha_2 a_3$ in the spectral sequence $E_r^{*,*}(BG_1)$.

Now, for the spectral sequence $E_r^{*,*}(BG_1)$, we compute the $E_3$-term up to degree $5$ and 
$E_r$-term up to degree $4$ for $r\geq 4$.
Since $d_2(z_1)=-\alpha_1 a_2$, 
the kernel of $d_2$ up to degree $5$ is
a free $\mathbb{Z}/\ell[a_2, z_2]$-module with the basis 
$\{1, b_2, b_2^2, a_3, b_3\}$
and the image of $d_2$ is $(a_2) \{ 1, b_2, b_2^2, a_3, b_3\}$.
So, the $E_3$-term is a free $\mathbb{Z}/\ell[z_2]$-module spanned 
by $1, b_2, b_2^2, a_3, b_3$ up to degree $5$.
Since $a_3 \not=0$, $a_3 z_2\not=0$, $a_3 b_2=0$ in $E_3^{*,*}$, 
the image of $d_3$ is spanned by 
\[
a_3=(-\alpha_2)^{-1} d_3(z_2)
\]
and the kernel of $d_3$ is spanned by 
\[
1, b_2, b_2z_2, b_2^2, a_3, b_3, 
\]
up to degree $4$, respectively.
Therefore, the $E_4$-term of the Leray-Serre spectral sequence for 
$H^{*}(B\Gone;\mathbb{Z}/\ell)$ 
up to degree $4$ is a graded vector space spanned by 
\[
1, b_2, b_2^2, 
b_2z_2, b_3.
\]
These generators are in $E_4^{*,q}$ $(q\leq 2)$, so that
these elements are in the kernel of $d_r$ for $r\geq 4$.
Therefore, up to degree $4$, the spectral sequence collapses at the $E_4$-level
and we obtain that
\[
H^i(B\Gone;\mathbb{Z}/\ell)=\left\{ \begin{array}{ll} 0 & \mbox{for $i=1$, }
\\
\mathbb{Z}/\ell & \mbox{for $i=0, 2, 3$, }
\\
\mathbb{Z}/\ell\oplus \mathbb{Z}/\ell & \mbox{for $i=4$.}
\end{array}
\right.
\]

Finally, we describe the induced homomorphism 
from the mod $\ell$ cohomology of the classifying space of $\Gone$ to
that of $A_3$.
The Leray-Serre spectral sequence for $H^{*}(BA_3;\mathbb{Z}/\ell)$ collapses at the 
$E_2$-level, so that 
\[
E_r^{*,*}(BA_3)=\mathbb{Z}/\ell[x_2, y_2, z_2]\otimes \Lambda(x_1, y_1, z_1),
\]
where $x_1, y_1 \in E_r^{1.0}(BA_3)$, $x_2, y_2 \in E_r^{2.0}(BA_3)$
and $z_1\in E_r^{0,1}(BA_3)$, $z_2\in E^{2,0}_r(BA_3)$.
The induced homomorphism of spectral sequences 
$\iota^*:E_\infty^{2,2}(B\Gone)\to E_\infty^{2,2}(BA_3)$
maps $b_2z_2$ to $x_1y_1z_2$.
Therefore, the induced homomorphism 
\[
\iota^*:H^{*}(B\Gone;\mathbb{Z}/\ell) \to H^{*}(BA_3;\mathbb{Z}/\ell)^W
\]
maps an element  representing $b_2z_2$ to $x_1y_1 z_2+\mbox{higher terms}$, 
which is non-zero in $H^{4}(BA_3;\mathbb{Z}/\ell)^W$.
By Proposition~\ref{proposition:odd-invariants}, 
$\dim_{\mathbb{Z}/\ell} H^{4}(BA_3;\mathbb{Z}/\ell)^W=1$.
Hence, the induced homomorphism above
 is an epimorphism.

\section{The case $\ell=2$}

Now, we deal with the case $\ell=2$.
For $\ell=2$, we define 
\[
\xi=\begin{pmatrix} -1 & 0 \\ 0 & -1 \end{pmatrix}, \quad 
\alpha=\begin{pmatrix}
 i & 0 \\ 0 & -i \end{pmatrix}, \quad
\beta=\begin{pmatrix} 0 &  i \\ i & 0 \end{pmatrix}, \quad 
\tau=\dfrac{1}{\sqrt{2}} \begin{pmatrix} 1 & i \\ i & 1
\end{pmatrix}.
\]
Then, by direct computation, we have the following proposition.
\begin{proposition}
$\alpha, \beta, \xi,\sigma, \tau$ are  unitary groups with determinant $1$. Moreover,
we have
\[
\beta^{-1} \alpha \beta=\xi \alpha;
\;
\sigma^{-1} \alpha\sigma=\alpha, \; \sigma^{-1} \beta \sigma=\alpha \beta; 
\;
\tau^{-1} \alpha\tau=\alpha\beta, \; \tau^{-1} \beta \tau=\beta.
\]
\end{proposition}
So, 
$\langle \alpha, \beta, \xi\rangle /\langle \xi\rangle$
is an elementary abelian $2$-group.
The matrixes representing the inner automorphisms defined by $\sigma$ and $\tau$ are
given by 
\[
(\alpha, \beta)^{\sigma} = (\alpha, \beta) \begin{pmatrix} 1 & 0 \\ 1 & 1 \end{pmatrix}, 
\quad
(\alpha, \beta)^{\tau}=(\alpha, \beta) \begin{pmatrix} 1 & 1 \\ 0 & 1 \end{pmatrix}.
\]
Theses matrixes generate the special linear group $SL_2(\mathbb{Z}/\ell)$. 
The ring of invariants $H^{*}(BA_2;\mathbb{Z}/2)^{ SL_2(\mathbb{Z}/2)}$
is known as Dickson invariants 
\[
\mathbb{Z}/2[u_{2}, u_{3}],
\]
where 
\begin{align*}
u_3
&=x_1y_1^2+x_1^2y_1,
\\
u_2
&=x_1^2+x_1y_1+y_1^2.
\end{align*}
Let us consider the subgroup $W$ generated by 
the inner automorphisms defined by
$\Gamma(\beta), \Delta(\sigma), \Delta(\tau)$ corresponding to 
the following matrixes:
\[
\begin{pmatrix} 1 & 0 & 0 \\ 0 & 1 & 0 \\ 1 & 0 & 1 \end{pmatrix}, \;
\begin{pmatrix} 1 & 1 & 0 \\ 0 & 1 & 0 \\ 0 & 0 & 1 \end{pmatrix}, \;
\begin{pmatrix} 1 & 0 & 0 \\ 1 & 1 & 0 \\ 0 & 0 & 1 \end{pmatrix}.
\]
It is equal to 
\[
\left. \left\{ \begin{pmatrix} a & b & 0 \\ c & d & 0 \\ * & 0 & 1 \end{pmatrix} 
\; \right| \; ad -bc =1 \right\}.
\]
We denote by $W_0$ the subgroup 
\[
\left. \left\{ \begin{pmatrix} a & b & 0 \\ c & d & 0 \\ 0 & 0 & 1 \end{pmatrix} 
\; \right| \; ad-bc =1 \right\}.
\]
The set of invariants $H^4(BA_3;\mathbb{Z}/2)^{W_0}$ is
a $\mathbb{Z}/2$ vector space spanned by 
$u_2^2$, $u_3 z_1$, $u_2z_1^2$, $z_1^4$.
As in the proof of Proposition~\ref{proposition:odd-invariants}, 
considering the homomorphism $f:A_3\to A_3$  induced by the inner automorphism
defined by $\Gamma(\beta)$, 
and the kernel of the induced homomorphism 
\[
1+f^*: H^4(BA_3;\mathbb{Z}/2)^{W_0} \to H^4(BA_3;\mathbb{Z}/2), 
\]
we have the following 
proposition.

%
%
\begin{proposition}
\label{proposition:even-invariants}
The set of invariants $H^4(BA_3;\mathbb{Z}/2)^{W}$ is
a $\mathbb{Z}/2$ vector space spanned by 
$u_2^2$, $u_3 z_1+u_2z_1^2+z_1^4$.
Moreover, 
\begin{align*}
&\; Q_1(u_3 z_1+u_2z_1^2+z_1^4)
\\
=&\; Q_0Q_1(x_1y_1z_1)
\\
=&\; x_1^4y_1^2z_1+
x_1^4y_1z_1^2+x_1^2y_1^4z_1+x_1^2y_1z_1^4+ 
x_1y_1^4z_1^2+x_1y_1^2z_1^4\not=0.
\end{align*}
\end{proposition}

Now, we end this paper by proving 
Propositions~\ref{proposition:projective} and \ref{proposition:bg1} 
for $\ell=2$.

\begin{proof}[Proof of Proposition~\ref{proposition:bg1} for $\ell=2$]
Since $PU(2)=SO(3)$, 
the mod $2$ cohomology of $BPU(2)$ is 
a polynomial algebra $\mathbb{Z}/2[v_2, v_3]$
and the induced homomorphism 
\[
H^2(BPU(2);\mathbb{Z}/2)\to H^2(BA_2;\mathbb{Z}/2)^{SL_2(\mathbb{Z}/2)}
\] is an isomorphism.
As in the odd prime case, 
let 
\[
\begin{array}{lll}
1 =1\otimes 1, & \quad a_2=v_2\otimes 1 + 1\otimes v_2, 
& \quad a_3=v_3\otimes 1 + 1\otimes v_3,
\\
& \quad b_2=v_1\otimes1, & \quad b_3=v_3\otimes 1.
\end{array}
\]
Then, 
the mod $2$ cohomology of 
$B(PU(2)\times PU(2))$ is also a polynomial algebra 
\[
\mathbb{Z}/2[a_2, a_3, b_2, b_3].
\]
The $E_2$-term $E_2^{*,*}(BG_1)$ is 
\[
\mathbb{Z}/2[a_2, a_3, b_2, b_3] \otimes \mathbb{Z}/2[z_1].
\]
The first non-trivial differential is given by
$d_2(z_1)=a_2$.
So, $E_3$-term is 
\[
\mathbb{Z}/2[ a_3, b_2, b_3] \otimes \mathbb{Z}/2[z_1^2].
\]
The next non-trivial differential is given by $d_3(z_1^2)=a_3$.
The $E_4$-term is 
\[
\mathbb{Z}/2[b_2, b_3]\otimes \mathbb{Z}/2[z_1^4].
\]
So, $\dim_{\mathbb{Z}/2} H^4(B\Gone;\mathbb{Z}/2)\leq 2$.
On the other hand, since $\dim_{\mathbb{Q}} H^4(B\Gone;\mathbb{Q})=2$, 
$d_r(z_1^4)=0$. Thus, the spectral sequence collapses at the $E_4$-level and 
we obtain 
\[
H^i(B\Gone;\mathbb{Z}/2)=\left\{ \begin{array}{ll} 0 & \mbox{for $i=1$, }
\\
\mathbb{Z}/2 & \mbox{for $i=0, 2, 3$, }
\\
\mathbb{Z}/2 \oplus \mathbb{Z}/2 & \mbox{for $i=4$.}
\end{array}
\right.
\]
As in the case that $\ell$ is an odd prime number, 
we consider the induced homomorphism
\[
\iota^*:H^{*}(B\Gone;\mathbb{Z}/2) \to H^{*}(BA_3;\mathbb{Z}/2).
\]
The spectral sequence for $H^{*}(BA_3;\mathbb{Z}/2)$ 
collapses at the $E_2$-level and
the induced homomorphisms
$\iota^{*}:E_\infty^{*, *}(B\Gone) \to E_\infty^{*,*}(BA_3)$ is
a monomorphism sending $b_2$, $z_1$ to $u_2$, $z_1$, respectively.
In particular, $\iota^*(b_2^2)=u_2^2$ and $\iota^*(z_1^4)=z_1^4$ in $E_\infty^{*,*}(BA_3)$.
It is clear that the image of the induced homomorphism
\[
\iota^{*}:H^4(B\Gone;\mathbb{Z}/2)\to H^4(BA_3;\mathbb{Z}/2)^W
\]
has dimension $2$. By Proposition~\ref{proposition:even-invariants},  
$\dim_{\mathbb{Z}/2} H^4(BA_3;\mathbb{Z}/2)^W=2$.
Hence, 
the induced homomorphism
above
 is an isomorphism.
\end{proof}

As in the proof of Proposition~\ref{proposition:projective} 
in \S5,
it is clear that the mod $2$ reduction 
\[
\Reduction:H^4(BG_1;\mathbb{Z})\to H^4(BG_1;\mathbb{Z}/2)
\]
is also an epimorphism. Therefore there exists an element 
 $u_4 \in H^4(BG_1;\mathbb{Z})$
such that $Q_1\Reduction(u_4)\not=0$ by Proposition~\ref{proposition:even-invariants}.
It completes the proof of 
 Proposition~\ref{proposition:projective}
 for $\ell=2$.


 \end{document}